\documentclass{article}
\usepackage{maa-monthly}

\theoremstyle{theorem}
\newtheorem{theorem}{Theorem}
 
\theoremstyle{definition}

\begin{document}

\title{A Trajectory from a Vertex to Itself on the Dodecahedron}
\markright{Trajectory on Dodecahedron}
\author{Jayadev S. Athreya and David Aulicino}

\maketitle

\begin{abstract}
We prove that there exists a geodesic trajectory on the dodecahedron from a vertex to itself that does not pass through any other vertex.
\end{abstract}

A straight-line trajectory on a the surface of a polyhedron is a straight line within a face that is uniquely extended over an edge so that when the adjacent faces are flattened the trajectory forms a straight line in the plane.  This is well-defined away from the vertices.  By choosing a tangent vector at a vertex, one can consider the corresponding straight-line trajectory emanating from that vertex.

These trajectories were considered in \cite{DavisDodsTraubYangGeodsRegTetCube, FuchsFuchsClosedGeodsRegPoly, FuchsGeodesicsRegPoly} where it was proven that there does not exist a straight-line trajectory on the tetrahedron, octahedron, cube, or icosahedron from a vertex to itself that does not pass through a different vertex before returning.  Fuchs \cite{FuchsGeodesicsRegPoly} speculated that such a trajectory might exist on the dodecahedron.

\begin{theorem}
\label{MainThm}
There exists a straight-line trajectory on the dodecahedron from a vertex to itself that does not pass through any other vertex.
\end{theorem}

\begin{proof}
Cut out the net of the dodecahedron in Figure \ref{DodecNetFig}.  Fold and tape it together. The straight line (red) diagonal trajectory is the desired trajectory.  The resulting trajectory on the dodecahedron is seen in Figure \ref{DodecFig}.
\begin{figure}[h]
\centering
\includegraphics[scale=.08]{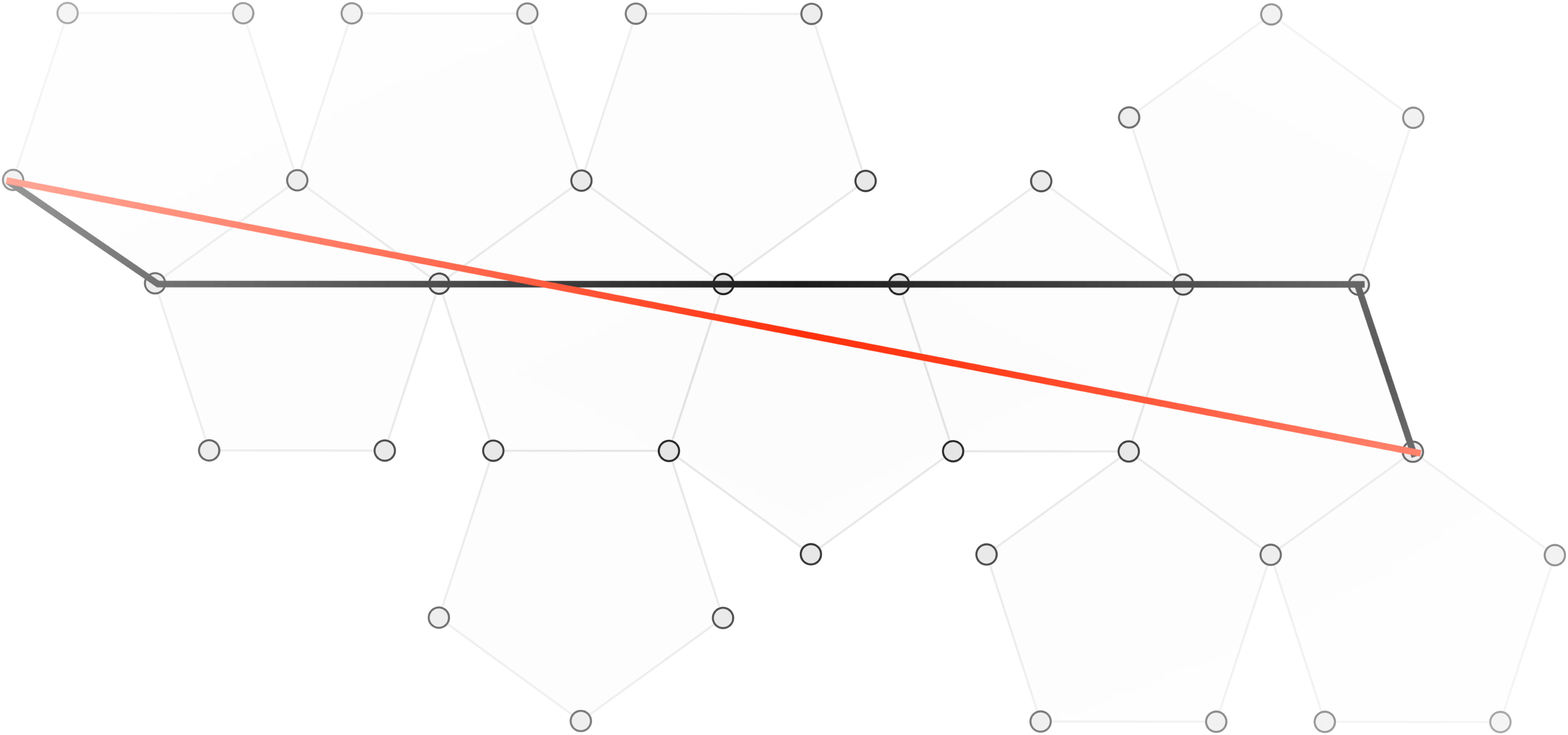}
\caption{A net of a dodecahedron with a straight-line trajectory.}
\label{DodecNetFig}
\end{figure}

\noindent If the starting point (the point on the left) is taken to be the origin, and the pentagons inscribed in a unit circle, then the terminal point (the point on the right) is $$\left( \frac 3 2 \sqrt{ \frac{5 - \sqrt 5}{2}} + 4\sqrt{ \frac{5 + \sqrt 5}{2}} ,  \frac1 4 \left(-1 - \sqrt 5\right)\right).$$ This can be seen via a direct calculation, most easily done by adding the vectors along the bold trajectory ($3$ horizontal diagonals, $2$ horizontal sides, and the initial and terminal sides).

\end{proof}

\begin{figure}[h]
\centering
\includegraphics[scale=.75]{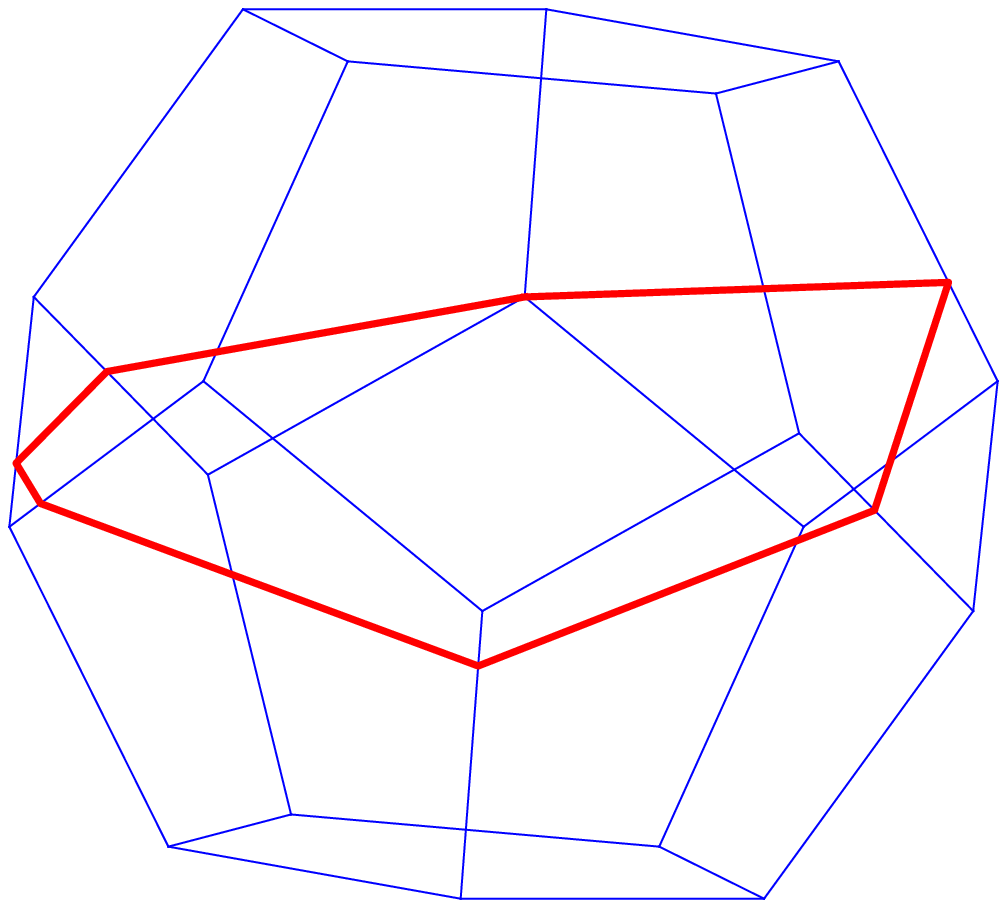}
\caption{A dodecahedron with a straight-line trajectory.}
\label{DodecFig}
\end{figure}

In forthcoming work, the authors use the theory of translation surfaces to give a uniform treatment for all Platonic solids and, in particular, classify all vertex-to-self trajectories on the dodecahedron.

\begin{acknowledgment}{Acknowledgments.} We would like to thank Samuel Leli\`evre for posing this problem to us.  We are very grateful to W. Patrick Hooper for many helpful discussions and for helping us generate Figure \ref{DodecFig} using the Sage-FlatSurf~\cite{SageFlatSurf} package.

J.S.A. gratefully acknowledges support from NSF CAREER grant DMS 1559860
and the D.A. from NSF grant DMS - 1738381 and PSC-CUNY Award \# 60571-00 48.
\end{acknowledgment}

\begin{biog}
\item[Jayadev S. Athreya]
\begin{affil}
Department of Mathematics, University of Washington, Box 354350, Seattle, WA, 98195-4530\\
jathreya@uw.edu
\end{affil}

\item[David Aulicino]
\begin{affil}
Department of Mathematics, Brooklyn College (CUNY), 2900 Bedford Avenue, Brooklyn, NY 11210-2889\\
david.aulicino@brooklyn.cuny.edu
\end{affil}
\end{biog}
\vfill\eject

\end{document}